\documentclass[11pt]{amsart}
\usepackage{epsfig,afterpage}
\usepackage{amsthm}
\usepackage{latexsym,amssymb,amsfonts}
\usepackage{amsmath}
\usepackage{graphicx}
\usepackage{color} % <- To write text in different colors.

\usepackage{comment}
\numberwithin{equation}{section}
\newtheorem{theorem}{Theorem}[section]
\newtheorem{lemma}{Lemma}[section]

\theoremstyle{remark}
\newtheorem*{remark}{Remark}

\def\card{\operatorname{Card}}
\def\diam{\operatorname{diam}}
\def\dim{\operatorname{dim}}
\def\area{\operatorname{area}}
\def\dist{\operatorname{dist}}
\def\max{\operatorname{max}}
\def\min{\operatorname{min}}
\def\dens{\operatorname{dens}}
\begin{document}
\title[Hausdorff measure of escaping sets]{Hausdorff measure of escaping sets on certain meromorphic functions}
\author{}
\address{}
\email{}
\author{Wenli Li}
\address{Mathematisches Seminar, Christian-Albrechts-Universit\"at zu Kiel,
24098 Kiel, Germany}
\email{w.li@math.uni-kiel.de}
\date{}
\begin{abstract}
We consider transcendental meromorphic function for which the set of finite singularities of its inverse is bounded. Bergweiler and Kotus gave bounds for the Hausdorff dimension of the escaping sets if the function has no logarithmic singularities over $\infty,$ the multiplicities of poles are bounded and the order is finite. We study the case of infinite order and find gauge functions for which the Hausdorff measure of escaping sets is zero or $\infty$.
\end{abstract}
\maketitle

\section{Introduction and main results}\label{introduction}

Suppose that $f$ is a meromorphic function on the whole complex plane. Denote by $f^{n}=f(f^{n-1})$ the $n$-th iterate of $f,$ for a natural number $n.$ The Fatou set $F(f)$ is defined as the set of all points with a neighborhood where the iterates $f^{n}$ of $f$ are defined and form a normal family. The Julia set $J(f)$ is the complement of $F(f),$ that is $J(f)=\hat{\mathbb{C}}\backslash F(f),$ where $\hat{\mathbb{C}}=\mathbb{C}\cup \left\{\infty\right\}$ and the escaping set of $f$ is
$$I(f)=\left\{z\in \mathbb{C}: f^{n}(z)\rightarrow \infty,~~~as~ n\rightarrow \infty\right\}.$$
It was shown that $I(f)\neq \emptyset$ and $J(f)=\partial I(f)$ by Eremenko \cite{Eremenko1989} for entire $f$ and by Dom\'inguez \cite{Dominguez1998} for meromorphic $f.$ We say that a meromorphic function $f$ is in the Eremenko-Lyubich class $\mathcal{B}$ if the set of finite singularities of its inverse function $f^{-1}$ is bounded. The result $I(f)\subset J(f)$ was proved for entire $f\in \mathcal{B}$ by Eremenko-Lyubich \cite{EremenkoLyubich1992} and by Rippon-Stallard \cite{RipponStallard1999} for meromorphic $f\in \mathcal{B}$. The Hausdorff dimension of Julia sets and related sets are studied in many papers, see e.g. \cite{KotusUrbanski2008, Stallard2008} for surveys. As a comprehensive introduction to iteration theory of meromorphic functions we refer the readers to \cite{Bergweiler1993}.

The order $\rho(f)$ of a meromorphic function is defined by
$$
\rho(f)=\limsup_{r\rightarrow \infty}\frac{\log T(r,f)}{\log r},
$$
where $T(r,f)$ denotes the Nevanlinna characteristic of $f,$ see \cite{Hayman1964,GoldenbergOstrovskii2008,Yang1995}.
Denote the Hausdorff dimension of a set $A\subset \mathbb{C}$ by $\dim (A)$ and  the two-dimensional Lebesgue measure of $A$ by $\area(A).$ For a subset $A\subset \mathbb{C}$ and a gauge function $h,$ we denote by $\mu_{h}(A)$ the Hausdorff measure of $A$ with respect to $h.$ The specific definition is given by \eqref{hmeasure} in the next section, where we also give more information about the gauge function.

Bara\'nski \cite{Baranski2008} and Schubert \cite{Schubert2007} proved that if an entire function $f\in \mathcal{B}$  and $\rho(f)<\infty,$ then $\dim(J(f))=2.$ Actually they proved that $\dim(I_{R}(f))=2$ for all $R>0,$ where
$$I_{R}(f)=\left\{z\in\mathbb{C}\colon \liminf_{n\rightarrow \infty }|f^{n}(z)|\geq R \right\},$$
and $I_{R}(f)\subset J(f)$ for large $R.$
It was pointed out by Bergweiler and Kotus in \cite{BergweilerKotus2012} that for meromorphic functions in $\mathcal{B}$ which have finite order and for which $\infty$ is an asymptotic value the same conclusion holds. Assume that $\infty$ is not an asymptotic value and that there exists $M\in \mathbb{N}$ such that the multiplicity of all poles, except finitely many, is at most $M.$ In the same paper they proved that for such a function, the Hausdorff dimension of its escaping set is no more than $2M\rho/(2+M\rho),$ where $\rho$ is the order of $f.$ 

If $f$ is as above but of infinite order, then the area of $I(f)$ is zero, yet there is an example \cite[section 6]{BergweilerKotus2012} with $\dim (I(f))=2$. McMullen \cite{McMullen1987} proved that the Julia set of $\lambda e^{z}$ has Hausdorff dimension two but in the presence of an attracting periodic cycle its area is zero.  He further remarked that  $\mu_{h}(J(\lambda e^{z}))=\infty$ for $h(r)=r^{2}\log^{n}(1/r),$ for arbitrary $n\in \mathbb{N}.$ Peter \cite{Peter2008} gave a fairly precise description of the gauge functions $h$ for which $\mu_{h}(J(\lambda e^{z}))= \infty.$

Analogously we aim in this paper to find a gauge function for which the Hausdorff measure of $I(f)$ is $0$ or $\infty$ for meromorphic functions in $\mathcal{B}$ of infinite order.

We shall use the $n$-th order
$$\rho_{n}(f)=\limsup_{r\rightarrow\infty}^{}\frac{\log_{+}^{n+1}T(r,f)}{\log r},~~~~~~~n\in \mathbb{N},$$
as a further discription of the growth rate (cf. \cite[Chapter 3]{JankVolkmann1985}). If we let $n=0,$ then $\rho_{0}(f)$ is what we defined previously as $\rho(f).$ And clearly we have $\rho(f)=\infty$ if $\rho_{n}(f)>0$ for $n\geq 1.$  
 
Not surprisingly, the representation of the gauge fuction corresponds to the growth rate of $f.$ Actually we choose 
\begin{equation}\label{h}
h(t)=t^{2}\left(\log_{}^{n} \frac{1}{t}\right)^{\gamma},
\end{equation} 
where $t\in (0, \delta_{n}],$ $\delta_{n}=1/\exp^{n}(\gamma)$ for $n\in\mathbb{N}$ and  $\gamma>0$ will depend on $\rho_{n}(f).$ 
More specifically, we obtain the following result.

\begin{theorem}\label{thm0}
	Let $f\in\mathcal{B}$ be a meromorphic function with $\rho_{}=\rho_{n}(f)$ satisfying $0<\rho<\infty.$ Suppose that $\infty$ is not an asymptotic value of $f$ and that there exists $M\in\mathbb{N}$ such that the multiplicity of all poles of $f,$ except possibly finitely many, is at most $M.$ If $h$ is given by \eqref{h} and  $\gamma< 2/(M\rho)$, then $\mu_{h}(I(f))= 0.$
\end{theorem}

The bound $2/(M\rho_{})$ is probably not sharp. However the following result shows that it cannot be replaced by any value greater than $8/(M\rho_{}).$

\begin{theorem}\label{thminfty}
Let $0<\rho_{}<\infty$ and $n,~M\in \mathbb{N}.$ Then there exists a meromorphic function $f\in\mathcal{B}$ of $n$-th order $\rho_{n}(f)=\rho_{}$ for which all poles have multiplicity $M$ and $\infty$ is not an asymptotic value such that if $h$ is as in \eqref{h} and $\gamma >8/(M\rho)$, then $\mu_{h}(I(f))=\infty.$
\end{theorem}

The paper is arranged as follows. In section \ref{gauge} we give some definitions and discuss some essential properties related to the gauge function. Afterwards we recall several lemmas which play an important role in our proof. Section \ref{proof1} is to give the  proof of Theorem \ref{thm0}. An example is constructed in section \ref{example} to prepare for the proof of Theorem \ref{thminfty} in section \ref{proof2}.

\section{Hausdorff Measure and gauge function}\label{gauge}
For $\alpha>0$ we say that $h\colon (0, \alpha]\rightarrow (0, +\infty)$ is a gauge function if it is continuous, increasing and satisfies $\lim_{t\rightarrow 0}h(t)=0.$ An example is the function that we defined in \eqref{h}.
For a set $A\subset\mathbb{C},$ we call a sequence $(A_{j})_{j\in\mathbb{N}}$ of sets  $A_{j}\subset\mathbb{C}$ a $\delta$-cover of $A$ if  $$A\subset\bigcup_{j=1}^{\infty}A_{j},$$and
$$\diam(A_{j})\leq \delta$$ for all $j\in\mathbb{N},$ 
where $\diam (A)= \sup_{x,y\in A}|x-y|$ denotes the diameter. The diameter with respect to  the spherical metric will be denoted by $\diam_{\chi}(A).$

Let $h$ be a gauge function. The measure $\mu_{h}$ defined by  
\begin{equation}\label{hmeasure}
\mu_{h}(A)=\lim_{\delta\rightarrow 0}\inf_{}\left\{\sum_{j=1}^{\infty}h(\diam A_{j})\colon (A_{j})_{j\in \mathbb{N}} ~\text{a}~\delta \text{-cover}~\text{of} ~A\right\},
\end{equation}
is called the Hausdorff measure corresponding to the function $h.$ For more details about Hausdorff measure we refer to Rogers \cite{Rogers1970} and Falconer \cite[chaper 2]{Falconer1997}. 

We are going to show some properties of interest for the gauge function that we choose, which also take part in our following proofs. First we prove the following results.
\begin{lemma}
Let $c>1$ and $h(t)$ be defined as in \eqref{h}, then
\begin{equation}\label{2v}
h(ct)\leq c^{2}h(t).
\end{equation}
\end{lemma}
\begin{proof}With the definition we have
\begin{align*}
h(ct)&=(ct)^{2}\left(\log^{n}\frac{1}{ct}\right)^{\gamma}\leq c^{2}t^{2}\left(\log^{n}\frac{1}{t}\right)^{\gamma}=c^{2}h(t).
\end{align*}
\end{proof}
\begin{lemma}\label{logpreaddtive}
Let $n\in \mathbb{N},$ $l\geq 1$ be a positive integer and $t_{1}, t_{2}, \cdots t_{l}$ be real numbers. If $0<t_{j}\leq 1/\exp^{n}2,$ $j=1,2,\cdots l,$ then we have
\begin{equation}\label{log}
\log_{}^{n}\left(\frac{1}{t_{1}t_{2}\cdots t_{l}}\right)\leq \left(\log_{}^{n}\frac{1}{t_{1}}\right)\left(\log_{}^{n}\frac{1}{t_{2}}\right)\cdots \left(\log_{}^{n}\frac{1}{t_{l}}\right).
\end{equation} 
\end{lemma}
\begin{proof}
We denote $t_{j_{0}}= \min\{t_{j}, 1\leq j \leq l\}.$ 
Consider first the case that $n=1$ and let $u_{1}=\log (1/t_{j_{0}}).$ 
Then we have
\begin{equation}\label{logl}
\log \frac{1}{t_{1}t_{2}\cdots t_{l}}
=\sum_{j=1}^{l}\log \frac{1}{t_{j}}
\leq u_{1}l.
\end{equation}
On the other hand noting that $1/t_{j}\geq\exp^{n} 2=\exp 2,$
\begin{equation}\label{logr}
\left(\log \frac{1}{t_{1}}\right)\left(\log \frac{1}{t_{2}}\right)\cdots \left(\log \frac{1}{t_{l}}\right)\geq 2^{l-1}\log \frac{1}{t_{j_{0}}}=u_{1}2^{l-1}.
\end{equation}
Since $l\leq 2^{l-1}$ for $l\geq 1,$ we deduce from \eqref{logl} 
and \eqref{logr} that \eqref{log} holds for $n=1$ and $t_{j}\leq e^{-2},$ $j=1,2,\cdots l,$ that is,
\begin{equation}\label{1}
\log^{}\left(\frac{1}{t_{1}t_{2}\cdots t_{l}}\right)\leq \left(\log^{}\frac{1}{t_{1}}\right)\left(\log^{}\frac{1}{t_{2}}\right)\cdots \left(\log^{}\frac{1}{t_{l}}\right).
\end{equation}
We now prove the conclusion by induction.
Suppose that \eqref{log} holds for $n=k$ and $t_{j}\leq 1/\exp^{k}2$ for $j=1,2,\cdots l,$ which is,
\begin{equation}\label{k}
\log_{}^{k}\left(\frac{1}{t_{1}t_{2}\cdots t_{l}}\right)\leq \left(\log_{}^{k}\frac{1}{t_{1}}\right)\left(\log_{}^{k}\frac{1}{t_{2}}\right)\cdots \left(\log_{}^{k}\frac{1}{t_{l}}\right).
\end{equation}
 Suppose that $t_{j}\leq 1/\exp^{k+1}2$ for $j=1,2,\cdots l.$ Therefore $1/\log^{k}(1/t_{j})\leq e^{-2}.$ Then from \eqref{k} and \eqref{1} we obtain
\begin{align*}
\log_{}^{k+1}\left(\frac{1}{t_{1}t_{2}\cdots t_{l}}\right)
&=\log_{} \log_{}^{k}\left(\frac{1}{t_{1}t_{2}\cdots t_{l}}\right)\\
&\leq \log_{} \left(\left(\log_{}^{k}\frac{1}{t_{1}}\right)\left(\log_{}^{k}\frac{1}{t_{2}}\right)\cdots \left(\log_{}^{k}\frac{1}{t_{l}}\right)\right)\\
&\leq \left(\log_{} \log_{}^{k} \frac{1}{t_{1}}\right)\left(\log \log^{k} \frac{1}{t_{2}}\right)\cdots \left(\log \log^{k} \frac{1}{t_{l}}\right)\\
&=\left(\log^{k+1}\frac{1}{t_{1}}\right)\left(\log^{k+1}\frac{1}{t_{2}}\right)\cdots \left(\log^{k+1}\frac{1}{t_{l}}\right),
\end{align*}
from which we see that \eqref{log} holds for $n=k+1$ if $t_{j}\leq 1/\exp^{k+1}2,$ $j=1,2,\cdots l.$  
\end{proof}
\begin{lemma}\label{productive}
Let $n, l \in \mathbb{N}$ and $l\geq 1.$ Set $h(t)$ as in \eqref{h}. Suppose that $t_{j}\leq 1/\exp^{n}2$, for $1\leq j \leq l$.  Then we have
\begin{equation}\label{mp}
h(t_{1}t_{2}\cdots t_{j})\leq \prod_{j=1}^{l} h(t_{j}).
\end{equation}
\end{lemma}
\begin{proof} 
Since $t_{j}\leq 1/\exp^{n}2$ for $1\leq j \leq l$ we deduce from \eqref{h} and \eqref{log} that
\begin{align*}
h(t_{1}t_{2}\cdots t_{l})
&=(t_{1}t_{2}\cdots t_{l})^{2} \left(\log^{n}\frac{1}{t_{1}t_{2}\cdots t_{l}}\right)^{\gamma}\\
&\leq (t_{1}t_{2}\cdots t_{l})^{2}\left(\left(\log^{n}\frac{1}{t_{1}}\right)\left(\log^{n}\frac{1}{t_{2}}\right)\cdots \left(\log^{n}\frac{1}{t_{l}}\right)\right)^{\gamma}\\
&=t_{1}^{2}\left(\log^{n}\frac{1}{t_{1}}\right)^{\gamma}t_{2}^{2}\left(\log^{n}\frac{1}{t_{2}}\right)^{\gamma}\cdots t_{l}^{2}\left(\log^{n}\frac{1}{t_{l}}\right)^{\gamma}\\
&=h(t_{1})h(t_{2})\cdots h(t_{l}).
\end{align*}
Therefore we obtain \eqref{mp}.
\end{proof}

\begin{lemma}\label{concave}
Suppose that $h(t)$ is defined as in \eqref{h} for $\gamma >0$ and $n\in\mathbb{N}.$ Define the function
$G(t)=h(\sqrt{t}).$ Then $G(t)$ is increasing and concave on $(0, \delta_{n}^{2}],$ where $\delta_{n}=1/\exp^{n}\gamma.$
\end{lemma}
\begin{proof}
According to the definition and \eqref{h} we have
$$G(t)= t\left(\log_{}^{n}\frac{1}{\sqrt{t}}\right)^{\gamma}.$$
Thus
\begin{align}
G'(t)&=\left(\log_{}^{n}\frac{1}{\sqrt{t}}\right)^{\gamma}+ t\left(\left(\log_{}^{n}\frac{1}{\sqrt{t}}\right)^{\gamma}\right)'\nonumber\\ 
&=\left(\log_{}^{n}\frac{1}{\sqrt{t}}\right)^{\gamma}- \frac{1}{2}\gamma\frac{\left(\log_{}^{n}\frac{1}{\sqrt{t}}\right)^{\gamma}}{\left(\log_{}^{n}\frac{1}{\sqrt{t}}\right)\left(\log_{}^{n-1}\frac{1}{\sqrt{t}}\right)\cdots \left(\log_{}^{}\frac{1}{\sqrt{t}}\right)}\nonumber\\ 
&=\left(\log_{}^{n}\frac{1}{\sqrt{t}}\right)^{\gamma}\left(1-\frac{\frac{\gamma}{2}}{\left(\log_{}^{n}\frac{1}{\sqrt{t}}\right)\left(\log_{}^{n-1}\frac{1}{\sqrt{t}}\right)\cdots \left(\log_{}^{}\frac{1}{\sqrt{t}}\right)}\right)\label{gprime}.
\end{align}
If $t\leq\delta_{n}^{2}= (1/\exp^{n}(\gamma))^{2}$ then 
 $$\left(\log^{n}\frac{1}{\sqrt{t}}\right)\left(\log^{n-1}\frac{1}{\sqrt{t}}\right)\cdots\log\frac{1}{\sqrt{t}}\geq \gamma \exp(\gamma)\cdots \exp^{n-1}(\gamma)\geq \gamma> \frac{\gamma}{2},$$
which yields with \eqref{gprime} that $$G'(t)\geq 0.$$
Hence $G(t)$ is increasing. One may also find that $G'(t)$ is decreasing on $(0,\delta_{n}^{2}]$ with a short observation of \eqref{gprime}.
And therefore $G(t)$ is a concave function on $(0, \delta_{n}^{2}].$
\end{proof}

\section{Notations and lemmas}\label{lemmas}
The following lemma is known as Iversen's theorem, see e.g. \cite[chapter 5]{GoldenbergOstrovskii2008}.
\begin{lemma}\label{infinitelypoles}
Let $f$ be a transcendental meromorphic function for which $\infty$ is not an asymptotic value. Then $f$ has infinitely many poles.
\end{lemma}
 We recall Koebe's theorem, which is usually stated only for univalent functions defined in the open unit disk, see \cite[Theorem 1.6]{Pommerenke1975}, but the following version follows immediately from this special case, see \cite[Lemma 2.1]{BergweilerKotus2012}. 
 
 For $a\in \mathbb{C}$ and $r>0$ we use the notation $D(a,r)=\left\{z\in \mathbb{C}\colon|z-a|<r\right\}$.
\begin{lemma}\label{koebe}
 Let $g\colon D(a,r)\rightarrow \mathbb{C}$ be univalent, $0<\lambda<1$ and $z\in D(a,\lambda r).$  Then
\begin{equation}\label{square}
\frac{\lambda}{(1+\lambda)^{2}}\leq \frac{|g(z)-g(a)|}{|(z-a)g'(a)|}\leq \frac{\lambda}{(1-\lambda)^{2}},
\end{equation}
\begin{equation}\label{cube}
\frac{1-\lambda}{(1+\lambda)^{3}}\leq \frac{|g'(z)|}{|g'(a)|} \leq \frac{1+\lambda}{(1-\lambda)^{3}},
\end{equation}
and
\begin{equation}\label{image}
g(D(a,r))\supset D\left(g(a),\frac{1}{4}|g'(a)|r\right).
\end{equation}
\end{lemma}
 Rippon-Stallard \cite[Lemma 2.1]{RipponStallard1999} proved the following result while Bergweiler-Kotus \cite[Lemma 2.2]{BergweilerKotus2012} made a supplement.
 
Designate $B(R)=\left\{z\in \mathbb{C}\colon |z|>R\right\} \cup \left\{\infty \right\}.$
\begin{lemma}\label{component}
 Let $f\in \mathcal{B}$ be transcendental. If $R>0$ such that sing$(f^{-1})\subset D(0,R),$ then all components of $f^{-1}(B(R))$ are simply-connected. Moreover, if $\infty$ is not an asymptotic value of $f$ then all components of $f^{-1}(B(R))$ are bounded and contain exactly one pole of $f.$
\end{lemma}
We continue with Jensen's inequality \cite[p.12]{Pachpatte2005}, one of the crucial tools used in our proof.
\begin{lemma}\label{Jensen}
Suppose that $I$ is an interval and the function $f\colon I\rightarrow \mathbb{R}$ is concave. For any points $x_{1}, x_{2},\cdots, x_{n}\in I$ and any real nonnegative numbers $r_{1}, r_{2}, \cdots, r_{n}$ such that $r_{1}+r_{2}+\cdots r_{n}= 1,$ we have
$$f\left(\sum_{j=1}^{n}r_{j}x_{j}\right)\geq \sum_{j=1}^{n}r_{j}f(x_{j}).$$
\end{lemma}
The next lemma from Jank-Volkmann \cite[p.103]{JankVolkmann1985} shows the relation between the $n$-th order and its number of poles for a meromorphic function.

\begin{lemma}\label{nrorder}
Suppose that $f$ is a meromorphic function and its $n$-th order is defined as in Section \ref{introduction} and that $n(r)$ denotes the number of the poles contained in the closed disc $\overline{D(0,r)}.$ Then we have
$$\rho_{n}(f)\geq\limsup_{r\rightarrow \infty}\frac{\log^{n+1}_{+}n(r)}{\log r}.$$
\end{lemma}

For $k\in\mathbb{N}$, let $N_{k}$ be a collection of disjoint compact sets in $\mathbb{R}^{n}$ such that\\
(a) every element of $N_{k}$ contains an element of $N_{k+1}$,\\
(b) every element of $N_{k+1}$ is contained in an element of $N_{k}$.

Let $\overline{N_{k}}=\bigcup_{A\in N_{k}}A$  and $N=\bigcap_{k=1}^{\infty}\overline{N_{k}}.$

McMullen \cite{McMullen1987} gave a lower bound for the Hausdorff dimension of a set $N$ constructed this way. Peter \cite[p.33]{Peter2008} used McMullen's method to obtain a sufficient condition for the set $N$ to have infinite Hausdorff measure with respect to some gauge function $h.$ We mention that they both worked with the Euclidean distance but the following lemma follows directly from the original one.

For measurable subsets $X, Y$ of the plane (or sphere) we define the Euclidean and the spherical  density of $X$ in $Y$ by
$$\dens(X, Y)=\frac{\area(X\cap Y)}{\area(Y)}~~~\text{and}~~~\dens_{\chi}(X,Y)=\frac{\area_{\chi}(X\cap Y)}{\area_{\chi}(Y)}.$$
Note that 
\begin{equation}\label{denschor}
\left(\frac{1+R^{2}}{1+S^{2}}\right)^{2}\dens(X, Y)\leq \dens_{\chi} (X,Y) \leq \left(\frac{1+S^{2}}{1+R^{2}}\right)^{2}\dens(X,Y),
\end{equation}
if $Y$ is a subset of the annulus $\{z\in \mathbb{C}\colon R<|z|<S\}.$

With this terminology Peter's result takes the following form.
\begin{lemma}\label{mass}
	For $k\in \mathbb{N,}$ let $N_{k},$ $N$ be as above. Suppose that $\Delta_{k}>0,$ $d_{k}>0,$ $d_{k}\rightarrow 0,$ such that if $B\in N_{k},$ then
	$$\dens_{\chi}(\overline{N_{k+1}},B)\geq \Delta_{k}~~~and~~~\diam_{\chi} B\leq d_{k}.$$
	Set $h(t)=t^{2}g(t)$ for $t>0,$ where $g(t)$ is a decreasing continuous function such that $h(t)$ is increasing and satisfies $\lim_{t\rightarrow0}t^{2}g(t)=0.$
Then we have $\mu_{h}(N)=\infty$ if
	$$\lim_{k\rightarrow \infty} g(d_{k})\prod_{j=1}^{k}\Delta_{j}=\infty.$$
\end{lemma}

\section{Proof of theorem \ref{thm0}}\label{proof1}
We follow the method used in \cite[Section 3]{BergweilerKotus2012} with some modifications.

With the assumption and Lemma \ref{infinitelypoles}, $f$ has infinitely many poles, say denoted by $a_{j}$ and ordered such that  $|a_{j}| \leq |a_{j+1}|$ for all $j\in \mathbb{N}.$ Let $m_{j}$ be the multiplicity of $a_{j}.$ Thus for some $b_{j}\in\mathbb{C}\backslash \left\{0\right\},$
$$f(z)\sim \left(\frac{b_{j}}{z-a_{j}}\right)^{m_{j}}~~~\text{as}~z\rightarrow a_{j}.$$
We may assume that $|a_{j}|\geq 1$ for all $j.$ Choose $R_{0}> 1$ such that sing$(f^{-1})\subset D(0, R_{0})$ and $|f(0)|< R_{0}.$

If $R\geq R_{0,}$ then all the components of $f^{-1}(B(R))$ are bounded and simply-connected and each component contains exactly one pole by Lemma \ref{component}. Let $U_{j}$ be the component containing $a_{j}.$ By the Riemann mapping theorem we may choose a conformal map
$$\phi_{j}\colon U_{j}\rightarrow D(0,R^{-1/m_{j}})$$
satisfying the normalization $\phi_{j}(a_{j})=0$ and $\phi_{j}'(a_{j})=1/b_{j}$, see \cite{BergweilerKotus2012} for the details.

Denote the inverse function of $\phi_{j}$ by $\psi_{j}.$ Since $\psi_{j}(0)=a_{j}$ and $\psi'_{j}(0)=b_{j}$ we can deduce from \eqref{image} that
\begin{equation}\label{subdisk}
U_{j}=\psi_{j}\left(D(0,R^{-1/m_{j}})\right)\supset D \left(a_{j},\frac{1}{4}|b_{j}|R^{-1/m_{j}}\right) \supset D \left(a_{j},\frac{1}{4R}|b_{j}|\right).
\end{equation}
Since $|f(0)|<R$ we have $0\notin U_{j}.$ Then \eqref{subdisk} implies that
$$\frac{1}{4R}|b_{j}| \leq |a_{j}|$$
for all $R\geq R_{0}.$ Hence
\begin{equation}\label{blessthan4a}
|b_{j}| \leq 4R_{0}|a_{j}|.
\end{equation}
Note that $\psi_{j}$ actually extends to a map univalent in $D(0,R_{0}^{-1/m_{j}}).$  Choosing $R\geq 2^{M}R_{0}$ we can apply \eqref{square} with
$$\lambda=(R/R_{0})^{-1/m_{j}}=(R_{0}/R)^{1/m_{j}}\leq \frac{1}{2}$$
and obtain
\begin{equation}\label{supdisk}
U_{j}\subset D\left(a_{j}, 2|b_{j}|R^{-1/M}\right),
\end{equation}
provided $j$ is so large that $m_{j}\leq M.$ With \eqref{subdisk} and \eqref{supdisk} we see that
$$D\left(a_{j},\frac{1}{4R}|b_{j}|\right) \subset U_{j} \subset D\left(a_{j},2R^{-1/M}|b_{j}|\right)$$
for large $j.$ Combining \eqref{blessthan4a} and \eqref{supdisk} and choosing $R\geq (16R_{0})^{M}$ we have
\begin{equation}\label{supradiusa}
U_{j}\subset D\left(a_{j}, \frac{1}{2}|a_{j}|\right) \subset D\left(0,\frac{3}{2}|a_{j}|\right).
\end{equation}
Let $n(r)$ denote the number of $a_{j}$ contained in the closed disc $\overline{D(0,r)}.$ Since the $U_{j}$ are pairwise disjoint we see with \eqref{subdisk} and \eqref{supradiusa} that
\begin{equation}\label{supofsumbj}
\sum_{j=1}^{n(r)}|b_{j}|^{2} \leq 36R^{2}r^{2},
\end{equation}
by comparing the areas of the domains (see \cite[p.5374]{BergweilerKotus2012}).

Let $D\subset B(R)\backslash\left\{\infty\right\}$ be a simply connected domain. Then any branch of the inverse of $f$ defined in a subdomain of $D$ can be continued analytically to $D.$ Let $g_{j}$ be a branch of $f^{-1}$ that maps $D$ to $U_{j}.$ Thus
\begin{equation}\label{g}
g_{j}(z)=\psi_{j}\left(\frac{1}{z^{\frac{1}{m_{j}}}}\right),
\end{equation}
for some branch of the $m_{j}$-th root.
Since we assumed that $R\geq 2^{M}R_{0}$ we deduce from \eqref{square} with $\lambda=1/2$ that
\begin{equation}\label{glambda}
|g'_{j}(z)|\leq \frac{12|b_{j}|}{m_{j}|z|^{1+\frac{1}{m_{j}}}}\leq \frac{12|b_{j}|}{|z|^{1+\frac{1}{M}}}
\end{equation}
for $z\in D\subset B(R)\backslash\{\infty\},$ provided $j$ is so large that $m_{j}\leq M.$ 
Moreover, if $U_{k}\subset B(R)$ with \eqref{supdisk} and \eqref{supradiusa} we have
\begin{align*}
\diam g_{j}(U_{k}) &\leq \sup^{}_{z\in U_{k}}|g'_{j}(z)|\diam U_{k}
\leq 
2^{1+\frac{1}{M}}12\frac{4}{R^{\frac{1}{M}}}\frac{|b_{j}|}{|a_{j}|^{1+\frac{1}{M}}}|b_{k}|.
\end{align*}
By induction if $U_{j_{1}},~U_{j_{2}},~...,~U_{j_{l}}\subset ~B(R),$ then with \eqref{supdisk} and transfering to the spherical distance
we have (see \cite[equation (3.10)]{BergweilerKotus2012})
\begin{equation}\label{diamchi}
\diam_{\chi}((g_{j_{1}} \circ g_{j_{2}} \circ ... \circ g_{j_{l-1}})(U_{j_{l}}))
\leq (2^{1+\frac{1}{M}}12)^{l-1}\frac{32}{R^{\frac{1}{M}}}\prod_{k=1}^{l}\frac{|b_{j_{k}}|}{|a_{j_{k}}|^{1+\frac{1}{M}}}.
\end{equation}
\\
Before we continue we shall prove the following result. This corresponds to \cite[lemma 3.1]{BergweilerKotus2012}, dealing with gauge functions of the form $h(t)=t^{\alpha}.$ These gauge functions are estimated using H\"older's inequality. Instead, here we consider  the gauge functions defined by \eqref{h} and use the results of section \ref{gauge} to estimate them.

\begin{lemma}\label{key}
Let $h$ be defined as in \eqref{h}. If $\gamma < 2/(M\rho_{})$
then 
\begin{equation}\label{converge}
\sum_{j=1}^{\infty}h\left(\frac{|b_{j}|}{|a_{j}|^{1+\frac{1}{M}}}\right)<\infty.
\end{equation}
\end{lemma}

\begin{proof}[Proof.]
For $l\geq0,$ we put$$P_{l}=\left\{ j\in \mathbb{N}\colon n(2^{l})\leq j< n(2^{l+1})\right\}=\left\{j\in\mathbb{N}\colon 2^{l}\leq |a_{j}| < 2^{l+1}\right\}.$$
Denote by $\card P_{l}$ the cardinality of $P_{l}$ and put
$$c_{j}=\left(\frac{|b_{j}|}{|a_{j}|^{L}}\right)^{2},$$
where $L=1+1/M.$
With \eqref{supofsumbj} we obtain
\begin{align*}
\sum_{j\in P_{l}} c_{j}&
= \sum_{j\in P_{l}}\frac{\left|b_{j}\right|^{2}}{\left|a_{j}\right|^{2L}}\\
&\leq 2^{-2lL}\sum_{j\in P_{l}}\left|b_{j}\right|^{2}\\
&\leq 2^{-2lL}\sum_{j=1}^{n(2^{l+1})}\left|b_{j}\right|^{2}\\
&\leq 2^{-2lL}36R^{2}~2^{2(l+1)}\\
\end{align*}
Thus  
\begin{equation}\label{cj}
\sum_{j\in P_{l}} c_{j}\leq K2^{-\frac{2l}{M}},
\end{equation}
where $K=144R^{2}.$
Set
$$S_{l}=\sum_{j\in P_{l}}h\left(\frac{|b_{j}|}{|a_{j}|^{1+\frac{1}{M}}}\right)$$ and $$G(t)= h(\sqrt{t}).$$ 
Then
\begin{equation}\label{sg}
S_{l}=\sum_{j\in P_{l}}G(c_{j}).
\end{equation}
\\
Let $\delta_{n}=1/\exp^{n}(\gamma)$ as in Lemma \ref{concave}.\\
\textbf{Case 1.} Suppose that 
$$\frac{K2^{-\frac{2l}{M}}}{\card P_{l}}\geq \delta_{n}^{2}.$$
Then 
$$\card P_{l}\leq \delta_{n}^{-2}K2^{-\frac{2l}{M}}<1,\qquad\text{as}\quad l\rightarrow \infty.$$
Thus $P_{l}=\emptyset$ for large $l.$ For such $l$ we have $S_{l}=0.$\\
\textbf{Case 2.} Suppose that $$\frac{K2^{-\frac{2l}{M}}}{\card P_{l}}<\delta_{n}^{2}.$$
Then from \eqref{cj} we have 
$$\frac{\sum_{j\in P_{l}}c_{j}}{\card P_{l}}\leq \frac{K2^{-\frac{2l}{M}}}{\card P_{l}}.$$
Hence by Lemma \ref{concave}, 
\begin{equation}\label{cg}
G\left(\frac{\sum_{j\in P_{l}}c_{j}}{\card P_{l}}\right) \leq G\left(\frac{K2^{-\frac{2l}{M}}}{\card P_{l}}\right).
\end{equation}
%From \eqref{blessthan4a} we have
%$$c_{j}=\frac{|b_{j}|^{2}}{|a_{j}|^{2L}}\leq \frac{16R_{0}^{2}|a_{j}|^{2}}{|a_{j}|^{2(1+\frac{1}{M})}}=\frac{16R_{0}^{2}}{\textcolor{blue}{|a_{j}|^{\frac{2}{M}}}}\leq \delta_{n}^{2}$$
%for large $j$ and thus 
%$$\frac{\sum_{j\in P_{l}}c_{j}}{\card P_{l}}\leq \delta_{n}^{2},$$
%for large $l.$
Applying Lemma \ref{Jensen} to $G(t)$ with $r_{j}= 1/ \card P_{l}$ and $x_{j}=c_{j}$ for $j\in P_{l}$ we obtain 
$$G\left(\frac{\sum_{j\in P_{l}}{} c_{j}}{\card~P_{l}}\right) \geq \frac{\sum_{j\in P_{l}}G\left(c_{j}\right)}{\card~P_{l}}.$$
This together with \eqref{cj}, \eqref{sg} and \eqref{cg} give,
\begin{align}
S_{l}&\leq (\card~P_{l})~
G\left(\frac{\sum_{j\in P_{l}} c_{j}}{\card~P_{l}}\right)\nonumber
\\ &
\leq (\card~P_{l})~G\left(\frac{K2^{-\frac{2l}{M}}}{\card~P_{l}}\right)\nonumber \\&
= (\card~P_{l})\frac{K2^{-\frac{2l}{M}}}{\card~P_{l}}\left(\log_{}^{n}\frac{1}{\sqrt{\frac{K2^{-\frac{2l}{M}}}{\card~P_{l}}}}\right)^{\gamma}\nonumber
\\&
=K2^{-2l/M}\left(\log_{}^{n}\frac{\sqrt{\card~P_{l}}}{\sqrt{K2^{-\frac{2l}{M}}}}\right)^{\gamma}.\label{sumg}
\end{align}
\\
Lemma \ref{nrorder} implies for $\varepsilon>0$ that
\begin{equation}\label{cardpl}
\card P_{l}\leq n(2^{l+1})\leq \exp^{n}\left((2^{l+1})^{\rho_{}+\varepsilon}\right),
\end{equation}
for large $l.$  
Then \eqref{sumg} and \eqref{cardpl} give,
\begin{align*}
S_{l}&\leq K2^{-\frac{2l}{M}}\left(\log^{n}\frac{\sqrt{\exp^{n}2^{(\rho_{}+\varepsilon)(l+1)}}}{K2^{-\frac{2l}{M}}}\right)^{\gamma}\\
&=K2^{-\frac{2l}{M}}\left(\log^{n-1}\left(\frac{1}{2}\exp^{n-1}2^{(\rho_{}+\varepsilon)(l+1)}-\log K+ \frac{2l}{M}\log 2\right)\right)^{\gamma}\\
&\leq K2^{-\frac{2l}{M}}\left(\log^{n-1}\left(\exp^{n-1}2^{(\rho_{}+\varepsilon)(l+1)}\right)\right)^{\gamma}\\
&\leq K2^{-\frac{2l}{M}}2^{((\rho_{}+\varepsilon)(l+1)+1)\gamma}\\
&=K2^{(\rho_{}+\varepsilon+1)\gamma}2^{-l(\frac{2}{M}-(\rho_{}+\varepsilon)\gamma)}
\end{align*}
for $n\geq 2$ and $l$ large.
If $\gamma<2/(M(\rho+\varepsilon)),$ then
$$\frac{2}{M}-(\rho_{}+\varepsilon)\gamma>0,$$
which implies the series $\sum_{l=0}^{\infty} S_{l}$ converges. The conclusion follows as $\varepsilon\rightarrow 0.$
\end{proof}

We continue the proof by denoting $E_{l}$ the collections of all components $V$ of $f^{-l}(B(R))$ for which $f^{k}(V)\subset U_{j_{k+1}}\subset B(R)$ for $k=0, 1, ..., l-1.$
For $V\in E_{l},$ there exist $j_{1}, j_{2},\cdots, j_{l}\geq n(R)$ such that
$$V=(g_{j_{1}}\circ g_{j_{2}}\circ \cdots \circ g_{j_{l-1}})(U_{j_{l}}).$$ 
From \eqref{diamchi} we have
\begin{equation}\label{v}
\diam_{\chi}(V)\leq (2^{1+\frac{1}{M}}12)^{l-1}\frac{32}{R^{1/M}}\prod_{k=1}^{l}\frac{|b_{j_{k}}|}{|a_{j_{k}}|^{1+\frac{1}{M}}}.
\end{equation}
It is easy to see from \eqref{blessthan4a} that for $R$ large,
$$2^{1+\frac{1}{M}}12\frac{|b_{j_{k}}|}{|a_{j_{k}}|^{1+\frac{1}{M}}}\leq \frac{1}{\exp^{n}2}$$
and
$$\frac{32}{R^{\frac{1}{M}}}\leq 2^{1+\frac{1}{M}}12.$$
Since there are $m_{j_{k}}$ branches of $f^{-1}$ mapping $U_{j_{k+1}}$ into $U_{j_{k}}$ for $k=1, 2, ..., l-1,$ we conclude that there are
$$\prod_{k=1}^{l-1}m_{j_{k}} \leq M^{l-1}$$
sets of diameters bounded as in \eqref{diamchi} which cover 
all those components $V$ of  
$f^{-l}(B(R))$ for which $f^{k}(V)\subset U_{j_{k+1}}\subset B(R)$ for $k=0, 1, ..., l-1.$

Now we may apply Lemma \ref{productive}, which together with \eqref{v} gives,
\begin{align*}
\sum_{V\in E_{l}}^{}h(\diam_{\chi}(V))&\leq M^{l-1}\sum_{j_{1}=n(R)}^{\infty}\ldots\sum_{j_{l}=n(R)}^{\infty} h\left((2^{1+\frac{1}{M}}12)^{l-1}\frac{32}{R^{\frac{1}{M}}}\prod_{k=1}^{l}\frac{|b_{j_{k}}|}{|a_{j_{k}}|^{1+\frac{1}{M}}}\right)
\\&
\leq M^{l-1}\sum_{j_{1}=n(R)}^{\infty}\ldots\sum_{j_{l}=n(R)}^{\infty} \prod_{k=1}^{l}h\left(2^{1+\frac{1}{M}}12\frac{|b_{j_{k}}|}{|a_{j_{k}}|^{1+\frac{1}{M}}}\right)
\\&
= \frac{1}{M}\left(M\sum_{j=n(R)}^{\infty} h\left(2^{1+\frac{1}{M}}12\frac{|b_{j_{}}|}{|a_{j_{}}|^{1+\frac{1}{M}}}\right)\right)^{l}
\end{align*}
for $R$ large enough.

We can get from \eqref{2v} and Lemma \ref{key} that if $\gamma<2/(M\rho),$
\begin{align*}
M\sum_{j=n(R)}^{\infty}h\left(2^{1+\frac{1}{M}}12\frac{|b_{j_{}}|}{|a_{j_{}}|^{1+\frac{1}{M}}}\right)\leq M\left(2^{1+\frac{1}{M}}12\right)^{2}\sum_{j=n(R)}^{\infty}h\left(\frac{|b_{j}|}{|a_{j}|^{1+\frac{1}{M}}}\right)<1,
\end{align*}
for $R$ large.
For such $R$ we find that
$$\lim_{l\rightarrow \infty}\sum_{V\in E_{l}}h(\diam_{\chi}(V))=0,\qquad \text{if} ~~~\gamma<\frac{2}{M\rho}.$$
We deduce from \eqref{supradiusa} that if $U_{j}\cap B(3R)\neq\emptyset,$ then $|a_{j}|>2R$ and $U_{j}\subset B(R).$ It follows that $E_{l}$ is a cover of the set
$$\left\{z\in B(3R)\colon f^{k}(z)\in B(3R)~\text{for}~1\leq k\leq l-1\right\}.$$
Therefore  
$$\mu_{h}(I_{3R}(f))=0,\qquad\quad\text{for}~\gamma<\frac{2}{M\rho}.$$
The conclusion follows since $I(f)=\bigcap_{R>0}I_{3R}(f)$.

\section{Construction of examples}\label{example}

Let $0<\rho<\infty$ and $n\in\mathbb{N}.$ We introduce the following function
\begin{align*}
q\colon [2^{1/ \rho}, \infty)\rightarrow [\exp^{n} 2, \infty),\qquad
q(r)=\exp^{n}(r^{\rho})
\end{align*}
and the inverse function 
\begin{equation}\label{pt}
p\colon [\exp^{n} 2, \infty)\rightarrow [2^{1/ \rho}, \infty),\qquad
p(t)=(\log^{n}t)^{1/ \rho}.
\end{equation}

We put $k_{0}=\lfloor\exp^{n} 2\rfloor+1.$ For $k\geq k_{0}$, $k\in\mathbb{N}$ set
\begin{equation}\label{nr}
n_{k}=\left\lfloor\frac{p(k)}{p'(k)}\right\rfloor.
\end{equation}
The next lemmas are giving some essential features of these functions, which help us to constuct the function in Theorem \ref{thmgz}.

\begin{lemma}\label{forgrowth}
\begin{equation}\label{tends0}
\frac{d}{dr}\left(\frac{q(r)}{q'(r)}\right)\rightarrow 0,\qquad
\frac{q(r)}{rq'(r)}\rightarrow 0,
\end{equation}
as $r \rightarrow \infty$.
\end{lemma}
\begin{proof}

By differentiation,
\begin{equation}\label{qpri}
q'(r)=(\exp^{n}r^{\rho})(\exp^{n-1}r^{\rho})\cdots(\exp^{}r^{\rho})\rho r^{\rho-1}.
\end{equation}
Thus for $n=1,$
$$\frac{q(r)}{rq'(r)}=\frac{1}{\rho r^{\rho}}\rightarrow 0,$$ 
$$\frac{d}{dr}\frac{q(r)}{q'(r)}=\frac{1-\rho}{\rho r^{\rho}}\rightarrow 0,$$
as $r\rightarrow \infty.$\\
For $n\geq 2,$
\begin{equation}\label{qq}
\frac{q(r)}{q'(r)}=\frac{1}{(\exp^{n-1}r^{\rho})(\exp^{n-2}r^{\rho})(\exp r^{\rho})\rho r^{\rho-1}},
\end{equation}
and clearly
$$\frac{q(r)}{rq'(r)}\rightarrow 0, \qquad r\rightarrow \infty.$$
Differentiating \eqref{qq} we obtain
\begin{align*}
\frac{d}{dr}\left(\frac{q(r)}{q'(r)}\right)&=-\frac{1}{\exp^{n-1}r^{\rho}}-\frac{1}{(\exp^{n-1}r^{\rho})(\exp^{n-2}r^{\rho})}-\cdots\\
&-\frac{1}{(\exp^{n-1}r^{\rho})(\exp^{n-2}r^{\rho})\cdots(\exp r^{\rho})}\left(1+\frac{\rho-1}{\rho r^{\rho}}\right).
%&=\frac{1}{\exp^{n-1}r^{\rho}}\left(1+\frac{1}{\exp^{n-2}r^{\rho}}+\frac{1}{(\exp^{n-2}r^{\rho})\cdots(\exp r^{\rho})}\left(1+\frac{1-\rho}{\rho r^{\rho}}\right)\right).
\end{align*}
It is easy to see that$$\frac{d}{dr}\left(\frac{q(r)}{q'(r)}\right)\rightarrow 0,\qquad r\rightarrow\infty.$$

\end{proof}

\begin{lemma}\label{pm} 
$$
p\left(t+\frac{1}{2}\right)-p(t)~~~\sim \frac{1}{2}p'(t),\qquad t\rightarrow\infty.
$$
\end{lemma}
\begin{proof}
From \eqref{pt} we have
\begin{equation}\label{dpt}
p'(t)=\frac{1}{\rho_{}}\left(\log_{}^{n}t\right)^{\frac{1}{\rho_{}}}\frac{1}{t(\log_{}t)(\log_{}^{2}t)\cdots(\log^{n}_{}t)},
\end{equation}
from which we can deduce that there exists $t_{0}$ such that $p'(t)$ is nonincreasing on $(t_{0}, +\infty).$ Therefore 
\begin{equation}\label{pt1}
\int_{t}^{t+\frac{1}{2}}p'(s)ds\leq \frac{1}{2}\max_{t\leq s\leq t+\frac{1}{2}}p'(s)=\frac{1}{2} p'(t),
\end{equation}
and 
\begin{equation}\label{p2}
\int_{t}^{t+\frac{1}{2}}p'(s)ds\geq \frac{1}{2}\min_{t\leq s \leq t+\frac{1}{2}}p'(t)=\frac{1}{2} p'\left(t+\frac{1}{2}\right)
\end{equation}
for $t\geq t_{0}.$ Note that
$$\frac{p'(t+\frac{1}{2})}{p'(t)}=\left(\frac{\log_{}^{n}(t+\frac{1}{2})}{\log_{}^{n}t}\right)^{\frac{1}{\rho_{n}}}\frac{t(\log_{}t)\cdots(\log_{}^{n}t)}{(t+\frac{1}{2})\log_{}(t+\frac{1}{2})\cdots \log^{n}_{}(t+\frac{1}{2})}\rightarrow 1$$
as $t\rightarrow \infty.$ 
Together with \eqref{p2} and \eqref{pt1} we have
$$p\left(t+\frac{1}{2}\right)-p(t)=\int_{t}^{t+\frac{1}{2}}p'(s)ds\sim \frac{1}{2}p'(t),\qquad t\rightarrow \infty.$$
\end{proof}

\begin{lemma}\label{sumnk}
For  $l\in \mathbb{R}$ with $\lfloor l \rfloor \geq k_{0}+1,$
$$\sum_{k=k_{0}+1}^{\lfloor l \rfloor}n_{k}\sim \int_{k_{0}+1}^{l}\frac{p(t)}{p'(t)}dt$$
as $l\rightarrow \infty.$
\end{lemma}
\begin{proof}Denote 
$$P(t)=\frac{p(t)}{p'(t)}\qquad \text{and}\qquad\int_{k_{0}+1}^{l}P(t)dt= I.$$
With \eqref{pt} and \eqref{nr} we get
\begin{equation}\label{np}
n_{k}\sim P(k)=\rho_{}k\log_{}k\cdots \log_{}^{n} k,
\end{equation}
for $k\geq k_{0}.$ Since $P(t)$ is increasing with $t$ we have for $k_{0}\leq k \leq l,$
$$P(k)=P(k)\int_{k}^{k+1}dt\leq \int_{k}^{k+1}P(t)dt.$$
Therefore
\begin{equation}\label{leq}
\sum_{k=k_{0}+1}^{\lfloor l \rfloor}P(k)\leq \int_{k_{0}+1}^{l+1}P(t)dt= I+\int_{l}^{l+1}P(t)dt\leq I+P(l+1).
\end{equation}
Similarly we obtain

\begin{equation}\label{geq}
\sum_{k=k_{0}+1}^{\lfloor l \rfloor}P(k)\geq \int_{k_{0}}^{l-1}P(t)dt\geq I+c_{0}-P(l),
\end{equation}
where $c_{0}=\int_{k_{0}}^{k_{0}+1}P(t)dt.$
From \eqref{np} we may take $l\geq k_{1}>k_{0}+1$ so large that $P(t)\geq t$ for all $t>k_{1}.$
Thus
$$I\geq \int_{k_{1}}^{l}tdt\geq \frac{1}{2}l^{2}-\frac{1}{2}k_{1}^{2}.$$
Since
$$P(l)=\rho l\log (l)\cdots \log^{n} (l)=o(l^{2})\qquad\text{as} \quad l\rightarrow\infty, $$
we have
$$P(l)=o(I),\qquad P(l+1)=o(I)$$
as $l\rightarrow \infty.$
Together with \eqref{np}, \eqref{leq} and \eqref{geq} we have our conclusion.
\end{proof}

\begin{lemma}\label{ppc} For $k_{0}\leq k<l,$ $k\in \mathbb{N}$ and $l\in\mathbb{R}, $
\begin{equation}\label{fraction}
\left(\frac{p(l)}{p(k)}\right)^{n_{k}}\geq\exp ( c\min\{k, l-k\}),
\end{equation}
where $c=(\log 2)^{n+1}/2.$ 
\end{lemma}

\begin{proof}
With \eqref{pt}, \eqref{nr} and \eqref{np} we have
\begin{align}
\left(\frac{p(l)}{p(k)}\right)^{n_{k}}&=\exp\left(n_{k}\log \frac{p(l)}{p(k)}\right)\nonumber\\
&\geq \exp\left(\frac{1}{2}\frac{p(k)}{p'(k)}\log \frac{p(l)}{p(k)}\right)\nonumber\\
&=\exp\left(\frac{1}{2}\rho k(\log k)\cdots (\log^{n}k)\log \left(\frac{\log^{n} l}{\log^{n} k}\right)^{\frac{1}{\rho}}\right)\nonumber\\
&= \exp \left(\frac{1}{2}k(\log k)\cdots (\log^{n} k) \log \frac{\log^{n}l}{\log^{n}k}\right)\label{ppnk},
\end{align}
We claim that
\begin{equation}\label{log2n}
k(\log k)\cdots (\log^{n} k) \log \frac{\log^{n}l}{\log^{n}k}\geq (\log 2)^{n+1}\min \{k, l-k\},
\end{equation}
which is verified as follows by induction to $n.$ We first consider that $n=0.$\\
\textbf{Case 1.}  If $k<\frac{l}{2}$ then
\begin{equation}\label{n11}
k\log\frac{l}{k}> k \log 2.
\end{equation}
\textbf{Case 2.} If $\frac{l}{2}\leq k < l,$ then $(l-k)/k\leq 1$ and thus
\begin{align}
k\log\frac{l}{k}&=k\log(1+\frac{l-k}{k})\nonumber\\
&\geq k\frac{l-k}{k}\log 2\nonumber\\
&=(l-k)\log 2.\label{n12}
\end{align}
Together \eqref{n11} and \eqref{n12} give
\begin{equation}\label{log0}
k\log\frac{l}{k}\geq (\log 2)\min\{k, l-k\},
\end{equation}
which is \eqref{log2n} for $n=0.$
Suppose that \eqref{log2n} holds for some $n.$\\
\textbf{Case 1.} If $\log^{n+1}l>2\log^{n+1}k$ and since $k\geq k_{0}$ then
\begin{align*}
k(\log k)\cdots (\log^{n+1}k)\log \frac{\log^{n+1}l}{\log^{n+1}k}
&> k(\log k)\cdots(\log^{n+1}k)\log 2 \nonumber\\
&> k(\log2)^{n+2}.
\end{align*}
\textbf{Case 2.} If $\log^{n+1}l \leq 2\log^{n+1}k$ then 
\begin{align*}
&k(\log k)\cdots (\log^{n+1}k)\log \frac{\log^{n+1}l}{\log^{n+1}k}\nonumber\\
&=k(\log k)\cdots (\log^{n+1}k)\log (1+\frac{\log^{n+1}l-\log^{n+1}k}{\log^{n+1}k})\nonumber\\
&\geq k(\log k)\cdots(\log^{n+1}k)\frac{\log^{n+1}l-\log^{n+1}k}{\log^{n+1}k}\log 2\nonumber\\
&=k(\log k)\cdots(\log^{n}k)\log\frac{\log^{n}l}{\log^{n}k}\log 2\nonumber\\
&\geq (\log 2)^{n+1}\min\{k, l-k\}\log 2\nonumber\\
&=(\log 2)^{n+2}\min\{k, l-k\}.
\end{align*}
Therefore 
\begin{equation}\label{log2n1}
k(\log k)\cdots (\log^{n+1} k) \log \frac{\log^{n+1}l}{\log^{n+1}k}\geq (\log 2)^{n+2}\min \{k, l-k\}.
\end{equation} 
From \eqref{log0} and \eqref{log2n1} we see that \eqref{log2n} holds with $n$ replaced by $n+1.$ Together with \eqref{ppnk} this gives \eqref{fraction}, by taking $c=(\log 2)^{n+1}/2$. 

\end{proof}

\begin{lemma}\label{ppck} For $k\in \mathbb{N},$ $l\in\mathbb{R}$ with $k>l\geq k_{0},$
\begin{equation}\label{fractionk}
\left(\frac{p(k)}{p(l)}\right)^{n_{k}}\geq\exp ( c(k-l)),
\end{equation}
where $c=(\log 2)^{n+1}/2.$
\end{lemma}
\begin{proof}
We prove along the same path as in Lemma \ref{ppc}. For $k>2l,$ instead of \eqref{ppnk} and \eqref{log2n} we have
\begin{equation}\label{kl}
\left(\frac{p(k)}{p(l)}\right)^{n_{k}}\geq \exp \left(\frac{1}{2}k(\log k)\cdots (\log^{n} k) \log \frac{\log^{n}k}{\log^{n}l}\right)
\end{equation}
and
\begin{equation}\label{kll}
k(\log k)\cdots (\log^{n} k) \log \frac{\log^{n}k}{\log^{n}l}>(k-l) (\log 2)^{n+1}.
\end{equation}
We first consider that $n=0.$\\
\textbf{Case 1.} If $k>2l,$
then $$k\log \frac{k}{l}>k\log 2>(k-l)\log 2.$$
\textbf{Case 2.} If $l<k\leq 2l,$
then
\begin{align*}
k\log\frac{k}{l}=k\log \left(1+\frac{k-l}{l}\right)\geq k\frac{k-l}{l}\log 2>(k-l)\log 2.
\end{align*}
Therefore we have \eqref{kll} for $n=0.$ Next we suppose that \eqref{kll} holds for some $n\in\mathbb{N}.$ 
\\
\textbf{Case 1.} If $\log^{n+1}k>2\log^{n+1}l,$ then
$$k(\log k)\cdots \log \frac{\log^{n+1}k}{\log^{n+1}l}>(k-l)(\log 2)^{n+2}.$$
\textbf{Case 2.} If $\log^{n+1}l<\log^{n+1}k\leq 2\log^{n+1}l,$
then with the assumption,
\begin{align*}
k(\log k)\cdots (\log^{n+1}k)\frac{\log^{n+1}k}{\log^{n+1}l}&=k(\log k)\cdots(\log^{n+1}k)
\log\left(1+\frac{\log^{n+1}k-\log^{n+1}l}{\log^{n+1}l}\right)\\
&\geq k(\log k)\cdots (\log^{n+1}k)\frac{\log^{n+1}k-\log^{n+1}l}{\log^{n+1}l}\log 2\\
&>k(\log k)\cdots (\log^{n}k)\log\frac{\log^{n}k}{\log^{n}l}\log 2\\
&>(\log 2)^{n+1}(k-l)\log 2\\
&=(k-l)(\log 2)^{n+2}.
\end{align*}
Hence  we have \eqref{kll} by induction. Together with \eqref{kl} and $c=(\log 2)^{n+1}/2$ we have \eqref{fractionk}.
\end{proof}

\begin{theorem}\label{thmgz}Let $p(k)$ and $n_{k}$ be defined by \eqref{pt} and \eqref{nr}. Put
\begin{equation}\label{gz}
g(z)=2\sum_{k=k_{0}+1}^{\infty}\frac{p(k)^{n_{k}}z^{n_{k}}}{z^{2n_{k}}-p(k)^{2n_{k}}}.
\end{equation}
Then $g\in \mathcal{B}$ and $\infty $  is not an aymptotic value of $g.$
\end{theorem}

\begin{remark}
Bergweiler and Kotus \cite{BergweilerKotus2012} gave an example for the case of infinite order, 
$$
f(z)=2\sum_{k=2}^{\infty}\frac{(\log k)^{n_{k}}z^{n_{k}}}{z^{2n_{k}}-(\log k)^{2n_{k}}},
$$
where $n_{k}=\lfloor k\log k \rfloor.$ Here we take $n_{k}=\lfloor \frac{p(k)}{p'(k)}\rfloor$ instead. If we let $n=1$ and $\rho_{}=1,$ then \eqref{gz} is essentially the above function. 
\end{remark}

\begin{proof}[Proof.] If $|z|\leq p(k)/2$, then
$$\left|\frac{p(k)^{n_{k}}z^{n_{k}}}{z^{2n_{k}}-p(k)^{2n_{k}}}\right|\leq \frac{|z|^{n_{k}}p(k)^{n_{k}}}{p(k)^{2n_{k}}-|z|^{2n_{k}}}\leq 2\frac{|z|^{n_{k}}}{p(k)^{n_{k}}}\leq 2^{1-n_{k}}.$$
From \eqref{np} we see that $n_{k}\geq k$ for large $k.$
Thus the series in \eqref{gz} converges locally uniformly and hence it defines a function  $g$ meromorphic in $\mathbb{C}.$ \\
Note that
\begin{equation}\label{ukl}
u_{k,l}=p(k)\exp\left(\frac{\pi il}{n_{k}}\right)
\end{equation}
are the poles of $g,$ where $k\in\mathbb{N}$ and $0\leq l \leq 2n_{k}-1.$ 
With $u_{k,l}$ we rewrite $g(z)$ as follows
$$g(z)=2\sum_{k=k_{0}+1}^{\infty}\sum_{l=0}^{2n_{k}-1}\frac{\nu_{k,l}}{z-u_{k,l}},$$
where 
\begin{align}
\nu_{k,l}=&\lim_{z\rightarrow u_{k,l}}(z-u_{k,l})2\frac{p(k)^{n_{k}}z^{n_{k}}}{z^{2n_{k}}-p(k)^{2n_{k}}}\nonumber \\
&=2p(k)^{n_{k}} u_{k,l}^{n_{k}}\lim_{z\rightarrow u_{k,l}}\frac{z-u_{k,l}}{z^{2n_{k}}-p(k)^{2n_{k}}}\nonumber \\
&=2p(k)^{n_{k}}u_{k,l}^{n_{k}}\lim_{z\rightarrow u_{k,l}}\frac{1}{2n_{k}z^{2n_{k}-1}}\nonumber \\
&=p(k)^{n_{k}}\frac{1}{n_{k}p(k)^{n_{k}-1}\left(\exp\left(\frac{\pi il}{n_{k}}\right)\right)^{n_{k}-1}}\nonumber\\
&=\frac{p(k)}{n_{k}}\exp\left(\frac{\pi i l (1-n_{k})}{n_{k}}\right)\label{nv}.
\end{align}
For $m\in\mathbb{N}$ we set $$W_{1}=\bigcup_{m\geq k_{0}+1}^{} \left\{z\colon~|z|=p\left(m+\frac{1}{2}\right)\right\}.$$
and for $\eta\in \mathbb{N},$ 
$$W_{2}=\bigcup_{m\geq k_{0}+1}^{}\bigcup_{\eta=1}^{2n_{m}} \left\{r\exp\left(\frac{i\pi(2\eta-1)}{2n_{m}}\right)\colon~ p\left(m-\frac{1}{2}\right)\leq r\leq p\left(m+\frac{1}{2}\right)\right\}.$$ 
Let $W=W_{1}\cup W_{2}.$ We will show that $g$ is bounded on this 'spider's web' $W$.
First let $z\in W_{1}$ and $m\in\mathbb{N}$ is taken such that $|z|=p(m+\frac{1}{2}).$
Noting that if $0<x<y,$ then 
\begin{equation}\label{xyz}
\frac{xy}{y^{2}-x^{2}}=\frac{x}{y-x}\frac{y}{y+x}\leq \frac{x}{y-x}.
\end{equation}
Since $p(k)$ is increasing with $k,$ from \eqref{gz} and \eqref{xyz} we have
\begin{align*}
\frac{1}{2}|g(z)|&\leq \sum_{k=k_{0}+1}^{m}\frac{p(k)^{n_{k}}|z|^{n_{k}}}{|z|^{2n_{k}}-p(k)^{2n_{k}}}+ \sum_{k=m+1}^{\infty}\frac{p(k)^{n_{k}}|z|^{n_{k}}}{p(k)^{2n_{k}}-|z|^{2n_{k}}}
\\&
\leq \sum_{k=k_{0}+1}^{m}\frac{p(k)^{n_{k}}}{|z|^{n_{k}}-p(k)^{n_{k}}}+\sum_{k=m+1}^{\infty}\frac{|z|^{n_{k}}}{p(k)^{n_{k}}-|z|^{n_{k}}}
\\&
=\sum_{k=k_{0}+1}^{m}\frac{1}{\left(\frac{p\left(m+\frac{1}{2}\right)}{p(k)}\right)^{n_{k}}-1}+\sum_{k=m+1}^{\infty}\frac{1}{\left(\frac{p(k)}{p\left(m+\frac{1}{2}\right)}\right)^{n_{k}}-1}
\\&
=\Sigma_{1,m}+\Sigma_{2,m}.
\end{align*}
From Lemma \ref{ppc} with $l=m+\frac{1}{2}$ we have, for $k_{0}\leq k \leq m,$
\begin{align*}
\left(\frac{p\left(m+\frac{1}{2}\right)}{p(k)}\right)^{n_{k}}\geq \exp\left(c\min\{k,m+\frac{1}{2}-k\}\right),
\end{align*}
where $c=(\log 2)^{n+1}/2.$
Thus
\begin{align*}
\Sigma_{1,m}&\leq \sum_{k=k_{0}+1}^{m}\frac{1}{\exp\left(c\min\{k,m+\frac{1}{2}-k\}\right)-1}\\
&=\sum_{k=k_{0}+1}^{\lfloor \frac{m}{2}\rfloor}\frac{1}{\exp(ck)-1}+\sum_{k=\lfloor \frac{m}{2}\rfloor+1}^{m}\frac{1}{\exp \left(c\left(m+\frac{1}{2}-k\right)\right)-1}\\
&=\sum_{k=k_{0}+1}^{\lfloor \frac{m}{2}\rfloor}\frac{1}{\exp(ck)-1}+\sum_{j=0}^{m-\lfloor \frac{m}{2}\rfloor-1}\frac{1}{\exp \left(c\left(j+\frac{1}{2}\right)\right)-1}\\
&\leq \sum_{k=1}^{\infty}\frac{1}{\exp(ck)-1}+\sum_{j=0}^{\infty}\frac{1}{\exp \left(c\left(j+\frac{1}{2}\right)\right)-1}~ = \colon C.
\end{align*}
Similarly with Lemma \ref{ppck} and $l=m+\frac{1}{2},$ we obtain
\begin{align*}
\Sigma_{2,m}&\leq \sum_{k=m+1}^{\infty}\frac{1}{\exp \left(c \left(k-\left(m+\frac{1}{2}\right)\right)\right)-1}=\sum_{j=0}^{\infty}\frac{1}{\exp \left(c\left(j+\frac{1}{2}\right)\right)-1}\leq C.
\end{align*}
 Therefore 
\begin{equation}\label{w1}
|g(z)|\leq 2(\Sigma_{1,m}+\Sigma_{2,m})\leq 4C,
\end{equation}
for $|z|=p\left(m+\frac{1}{2}\right)$, $m\in\mathbb{N}.$ 
Next we consider $z\in W_{2}.$ Then $z=r\exp(i\pi(2\eta-1)/(2n_{m})),$ where $p(m-\frac{1}{2})\leq r \leq p(m+\frac{1}{2})$ and $\eta\in\mathbb{N},$ $1\leq \eta \leq 2n_{m}.$ Thus
$$z^{2n_{m}}=\left(r\exp\left(i\pi\frac{2\eta-1}{2n_{m}}\right)\right)^{2n_{m}}=-r^{2n_{m}}$$
With this, \eqref{xyz} and since $p(k)$ is increasing with $k,$ we have 
\begin{align*}
\frac{1}{2}|g(z)|&\leq \sum_{k=k_{0}+1}^{m-1}\frac{p(k)^{n_{k}}r^{n_{k}}}{r^{2n_{k}}-p(k)^{2n_{k}}}+\frac{p(m)^{n_{m}}r^{n_{m}}}{r^{2n_{m}}+p(m)^{2n_{m}}} +\sum_{k=m+1}^{\infty}\frac{p(k)^{n_{k}}r^{n_{k}}}{p(k)^{2n_{k}}-r^{2n_{k}}}
\\&
\leq\sum_{k=k_{0}+1}^{m-1}\frac{p(k)^{n_{k}}}{p(m-\frac{1}{2})^{n_{k}}-p(k)^{n_{k}}} +2 + \sum_{k=m+1}^{\infty}\frac{p(m+\frac{1}{2})^{n_{k}}}{p(k)^{n_{k}}-p(m+\frac{1}{2})^{n_{k}}}
\\&
\leq \Sigma_{1,m-1}+2 + \Sigma_{2,m}\\&
\leq 2C+2.
\end{align*}
Combining with \eqref{w1} it follows that
\begin{equation}\label{bounded}
|g(z)|\leq4C+4~\text{for}~ z\in W.
\end{equation}
Actually $g$ is bounded on a larger set, which we want to show next.\\
From Lemma \ref{pm} we have
\begin{equation}\label{pri}
p\left(m+\frac{1}{2}\right)-p(m)=\int_{m}^{m+\frac{1}{2}}p'(t)dt \sim \frac{1}{2}p'(m),\qquad m\rightarrow \infty.
\end{equation}
And note that by \eqref{nr}
\begin{equation}\label{pi}
|u_{m,\eta}-u_{m,\eta+1}|=p(m)\left|\exp\left(\frac{i\pi}{n_{m}}\right)-1\right|\sim p(m)\frac{\pi}{n_{m}}\sim \pi p'(m),\qquad m\rightarrow\infty.
\end{equation}
If $W_{m,\eta}$ denotes the component of $\mathbb{C}\diagdown W$ that contains $u_{m,\eta}$, we find that there exists $\lambda>0$ such that
\begin{equation}\label{distb}
\dist\left(u_{m,\eta},\partial W_{m,\eta} \right)\geq 2\lambda p'(m),
\end{equation}
for $m$ large and $\eta\in \left\{0,1,...,2n_{m}-1\right\}.$ \\
Consider the function
\begin{equation}\label{zeta}
\zeta(z)= g(z)-\frac{\nu_{m,\eta}}{z-u_{m,\eta}},
\end{equation}
which is holomorphic in the closure of $W_{m,\eta}.$ For $z\in \partial W_{m,\eta}$ with \eqref{nv}, \eqref{bounded}, \eqref{nr} and $n_{m}\geq p(m)/(2p'(m))$ we have
\begin{align*}
|\zeta(z)|&\leq |g(z)|+\frac{|\nu_{m,\eta}|}{|z-u_{m,\eta}|}\\
&\leq 4C+4+\frac{p(m)}{2\lambda n_{m}p'(m)}\\
&\leq 4C+4+\frac{1}{\lambda},
\end{align*}
for $m$ large. By the maximum principle,
$$|\zeta(z)|\leq 4C+ 4 + \frac{1}{\lambda}~\text for~ z\in W_{m,\eta}.$$
We put $r_{m}=\lambda p'(m)$ and deduce that if $z\in W_{m,\eta}\diagdown D(u_{m,\eta}, r_{m}),$ then
\begin{align*}
|g(z)|&\leq |\zeta(z)|+\frac{|\nu_{m,n}|}{r_{m}}\\
&\leq 4C+4+\frac{1}{\lambda}+ \frac{p(m)}{\lambda n_{m}p'(m)}\\
&\leq 4C+ 4 + \frac{3}{\lambda}.
\end{align*}
This means that $g$ is large only in small neighborhoods of the poles.

On the other hand we will show that the set of critical values of $f$ is bounded by verifying that there are  no critical points of $g$ in these small neighborhoods of the poles. 

Assume that $z\in \partial W_{m,\eta}$ and $m', \eta'$ are such that $z\in \partial W_{m',\eta'}.$ Then $|m-m'|\leq 1$ and so $r_{m}\leq 2r_{m'}$ by \eqref{distb}. Therefore $ \overline{D\left(z,\frac{1}{2}r_{m}\right)}\cap D\left(u_{m',\eta'}, r_{m'}\right)=\emptyset.$ Thus
\begin{align*}
|g'(z)|&=\frac{1}{2\pi}\left|\int_{|\xi-z|=\frac{1}{2}r_{m}}\frac{g(\xi)}{(\xi-z)^{2}}d\xi\right|\leq \frac{r_{m}}{2}\max_{|\xi-z|=\frac{1}{2}r_{m}}\frac{|g(\xi)|}{|\xi-z|^{2}}\leq \frac{2}{r_{m}}\left(4C+4+\frac{3}{\lambda}\right).
\end{align*}

Since $n_{m}>p(m)/(2p'(m))$ for $m$ large, from \eqref{zeta}, \eqref{nv} and \eqref{distb} we have  
\begin{align*}
|\zeta'(z)|&\leq |g'(z)|+ \frac{|\nu_{m,\eta}|}{|z-u_{m,\eta}|^{2}}\\
&\leq \frac{2}{r_{m}}\left(4C+4+\frac{3}{\lambda}\right)+\frac{p(m)}{n_{m}(2\lambda p'(m))^{2}}\\
&= \frac{2}{r_{m}}\left(4C+4+\frac{3}{\lambda}\right)+\frac{p(m)}{4\lambda n_{m} r_{m}p'(m)}\\
&\leq \frac{2}{r_{m}}\left(4C+4+\frac{3}{\lambda}\right)+\frac{1}{2\lambda r_{m}}\\
&\leq \frac{2}{r_{m}}\left(4C+4 + \frac{13}{4\lambda}\right)
\end{align*}
for $z\in \partial W_{m,\eta}.$ It implies  by maximum principle that
$$|\zeta'(z)|\leq \frac{2}{r_{m}}\left(4C+4 + \frac{13}{4\lambda}\right)~~~\text{for} ~ z\in W_{m,\eta}.$$
Choose $\delta>0$ sufficiently small and $z\in D\left(u_{m,\eta}, \delta r_{m}\right).$ Since $n_{m}<2p(m)/p'(m)$ for $m$ large we have
\begin{align*}
|g'(z)|&\geq \frac{|\nu_{m,\eta}|}{|z-u_{m,\eta}|^{2}}-|\zeta'(z)|\\
&\geq \frac{p(m)}{\delta^{2}n_{m}r_{m}\lambda p'(m)}-\frac{2}{r_{m}}\left(4C+4+ \frac{13}{4\lambda}\right)\\
&\geq \frac{2}{r_{m}}\left(\frac{1}{4\delta^{2}\lambda}-4C-4-\frac{13}{4\lambda}\right) > 0.
\end{align*}
Hence if $g'(z)=0$ for some $z\in W_{m,\eta}$ then $|z-u_{m,\eta}|\geq \delta r_{m}.$ Therefore
\begin{equation}\label{cbd}
|g(z)|\leq |\zeta(z)|+ \frac{|\nu_{m,\eta}|}{|z-u_{m,\eta}|}\leq 4C+4+\frac{1}{\lambda}+\frac{p(m)}{\delta \lambda n_{m}p'(m)} \leq 4C+4+\frac{1}{\lambda}+\frac{2}{\delta\lambda}
\end{equation}
 as claimed. The same is true for the set of asymptotic values of $g$ with \eqref{bounded}. Hence $g\in \mathcal{B}.$ 
\end{proof}

\begin{theorem} Let $g$ be defined as in \eqref{gz}. Then $\rho_{n}(g)=\rho.$
\end{theorem}
\begin{proof}
From Lemma \ref{sumnk} the number $n(r,g)$ of poles of $g$ in $\overline{D(0,r)}$ satisfies
$$
n(r,g)=\sum_{k=k_{0}+1}^{\left\lfloor q(r)\right\rfloor}2n_{k}\sim 2\int_{k_{0}+1}^{ q(r)}\frac{p(t)}{p'(t)}dt.
$$
Now let $t=q(s)$ and $r_{0}=p(k_{0}+1).$ Then
\begin{align*}
\int_{k_{0}+1}^{q(r)}\frac{p(t)}{p'(t)}dt &=\int_{r_{0}}^{r}\frac{p(q(s))}{p'(q(s))}q'(s)ds
\\&
=\int_{r_{0}}^{r}\frac{s}{p'(q(s))}q'(s)ds
\\&
=\int_{r_{0}}^{r}q'(s)^{2}sds.
\end{align*}
Hence
\begin{equation}\label{nrestimate}
n(r,g)\sim 2\int_{r_{0}}^{r}q'(s)^{2}sds
\end{equation}
for $r\rightarrow \infty.$ By Lemma \ref{forgrowth} we have
\begin{equation}\label{secint}
\frac{q(r)q'(r)}{q'(r)^{2}r}=\frac{q(r)}{rq'(r)}\rightarrow 0,\qquad r\rightarrow \infty.
\end{equation}
On the other hand,
\begin{align*}
\frac{q''(r)q(r)}{q'(r)^{2}}&= 1-\frac{q'(r)q'(r)-q(r)q''(r)}{q'(r)^{2}}
\\&
=1-\frac{d}{dr}\left(\frac{q(r)}{q'(r)}\right).
\end{align*}
Again with lemma \ref{forgrowth} we have
\begin{equation}\label{firint}
\frac{q(r)q''(r)}{q'(r)^{2}}\rightarrow 1,~~~ r\rightarrow \infty,
\end{equation}
We claim that
\begin{equation}\label{middle}
2\int_{r_{0}}^{r}q'(s)^{2}sds\sim q(r)q'(r)r~~~~~~\text{as}~r\rightarrow \infty.
\end{equation}
In fact with \eqref{secint} and \eqref{firint} by l'Hospital's rule we have
\begin{align*}
\lim_{r\rightarrow \infty}\frac{q(r)q'(r)r}{\int_{r_{0}}^{r}q'(s)^{2}s ds}&=\lim_{r\rightarrow \infty}\frac{\frac{d}{dr}(q(r)q'(r)r}{\frac{d}{dr}\int_{r_{0}}^{r}q'(s)^{2}s ds}\\
&=\lim_{r\rightarrow \infty}\frac{q'(r)^{2}r+q(r)q''(r)r+q(r)q'(r)}{q'(r)^{2}r}\\
&=\lim_{r\rightarrow \infty}\left(1+\frac{q(r)q''(r)}{q'(r)^{2}}+\frac{q(r)}{q'(r)r}\right)\\
&=2.
\end{align*}
Therefore from \eqref{nrestimate}, \eqref{middle} and the definition of counting function,
$$N(r,g) \sim \frac{1}{2}q^{2}(r)$$
as $r \rightarrow \infty.$

Suppose that $r$ has the form $r=p(k+\frac{1}{2})$ for  $k_{0}\leq k\in \mathbb{N}$ large. From \eqref{bounded} we have $m(r,g)\leq 4C+4.$  Since
$$
T(r,g)= N(r,g) + m(r,g)
$$
we obtain 
\begin{align*}
T(r,g)\sim \frac{1}{2}q(r)^{2},\qquad r \rightarrow \infty.
\end{align*}
It yields that
\begin{align*}
\log T(r,g)&=(1+o(1))2\log q(r)\\
&=(1+o(1))2 \exp^{n-1} (r^{\rho}),
\end{align*}
and thus
\begin{equation}\label{lin}
\log^{n+1}T(r,g)\sim \rho \log r
\end{equation}
as $r\rightarrow\infty$ through $r$-values of the form $r=p\left(k+\frac{1}{2}\right).$
It follows that \eqref{lin} holds for all $r$ since $T(r,g)$ is increasing with $r.$ Hence
\begin{align*}
\rho_{n}(g)=\limsup_{r\rightarrow \infty}\frac{\log^{n+1}T(r,g)}{\log r}=\rho.
\end{align*}
\end{proof}

\section{Proof of Theorem \ref{thminfty}}\label{proof2}

Let $g$ be the function constructed in section \ref{example} and put $f(z)=g(z)^{M}$. Hence the multiplicity of all poles of $f$ is $M,$ $f\in \mathcal{B}$ without $\infty$ as its aymptotic value and $\rho_{n}(f)=\rho.$ 

As in section  \ref{proof1} we denote the sequence of poles by $a_{j},$ ordered such that $|a_{j}|\leq |a_{j+1}|$ for all $j\in \mathbb{N}.$
Choose $b_{j}$ as in section \ref{proof1} so that
$$f(z)\sim \left(\frac{b_{j}}{z-a_{j}}\right)^{M}\qquad \text{as}~~~z \rightarrow a_{j},$$
for each $j\in \mathbb{N}.$ We thus have $a_{j}=u_{m,\eta}$ and $b_{j}=\nu_{m,\eta}$ for some $m, \eta \in \mathbb{N}$ and $0\leq \eta \leq 2n_{m}-1.$ 
 
Choose $R_{0}\geq 4C+4+\frac{1}{\lambda}+\frac{2}{\delta\lambda},$ where $\lambda,~ \delta$ are as in \eqref{cbd} and $R_{l}=R_{0}\exp(2^{l})$ for $l\in\mathbb{N}.$  We denote by $E_{l}$ the collections of  all components $V$ of $f^{-l}(B(R_{l}))$ which satisfy $f^{k}(V)\subset U_{j_{k+1}}\subset B(R_{k})$ for $0\leq k \leq l-1$ and $\overline{E}_{l}=\bigcup_{A\in E_{l}}A.$ It follows  that $E=\bigcap_{l=1}^{\infty}\overline{E_{l}}\subset I(f).$

The estimates obtained in section \ref{proof1} also hold with $R$ replaced by $R_{l}.$ So we may use them for the map $g_{j}$ that maps $D\subset B(R_{l})\backslash \{\infty\}$ to $U_{j},$ the component of $f^{-1}(B(R_{l}))$ containing $a_{j}.$
From \eqref{diamchi} we deduce that if $V\in E_{l}$ such that $f^{k}(V)\subset U_{j_{k+1}}\subset B(R_{k})$ for $0\leq k\leq l-1,$ then \eqref{v} holds. 

Here $a_{j_{k}}$ is a pole of $f$ that is contained in $U_{j_{k}}$ for $k=1,2,\cdots,l.$ From \eqref{ukl} and \eqref{nv} we know that $|a_{j_{k}}|=:r_{j_{k}}=p(l)$ for some $l\geq k_{0}+1$ and accordingly
$|b_{j_{k}}|=r_{j_{k}}/n_{j_{k}}.$ With the definition of $p,$ $q$ and $n_{j_{k}}$ we have
$$|b_{j_{k}}|\sim \frac{p(l)}{p(l)/p'(l)}=p'(l)=\frac{1}{q'(p(l))}.$$
Therefore
\begin{equation}\label{abr}
\frac{|b_{j_{k}}|}{|a_{j_{k}}|^{1+\frac{1}{M}}}\sim \frac{1}{q'(r_{j_{k}})r_{j_{k}}^{1+\frac{1}{M}}}.
\end{equation}
Recall that $q(r)=\exp^{n}(r^{\rho_{}})$ is convex and thus $q'(r)$ is increasing. Moreover $r_{j_{k}}\geq R_{k-1}$ for $k=1,2,\cdots,l.$ It follows from \eqref{v} and \eqref{abr} that
\begin{equation}\label{dl}
\diam_{\chi}(V)\leq \prod_{k=1}^{l}\frac{A}{q'(R_{k-1})R_{k-1}^{1+\frac{1}{M}}}=:d_{l},
\end{equation}
where $A\neq 0$ is a constant.

With $d_{l}$ we intend to apply Lemma \ref{mass}. In order to do so we are estimating $\Delta_{l}.$ From \eqref{subdisk} and \eqref{bounded} we deduce that
\begin{equation}\label{in}
D\left(a_{j},\frac{|b_{j}|}{4R_{l}^{\frac{1}{M}}}\right)\subset U_{j}= W_{m,\eta}\cap f^{-1}(B(R_{l})). 
\end{equation}
Meanwhile \eqref{pri} and \eqref{pi} imply that
\begin{equation}\label{ex}
W_{m,\eta}\subset D(u_{m,\eta},\tau|\nu_{m,\eta}|)=D\left(a_{j}, \tau|b_{j}|\right),
\end{equation}
where $\tau=1/2+\pi/2,$ $a_{j}=u_{m,\eta},$ and $m$ large. 
%Therefore
%$$\dens\left(f^{-1}(B(R_{1})),W_{m,\eta}\right)\geq \frac{1}{16\tau^{2}R_{1}^{2/M}}.$$
For $\varepsilon>0$ small set
$$A(S)=\left\{z\in \mathbb{C} \colon S<|z|<2S\right\},$$
$$A_{\varepsilon}(S)=\left\{z\in \mathbb{C} \colon (1+\varepsilon)S<|z|<(1-\varepsilon)
2S\right\}.$$
Then from \eqref{in} we have
\begin{align*}
\area \left(\bigcup_{j\in\mathbb{N}}U_{j}\cap A(S)\right)&\geq \area\left(\bigcup_{a_{j}\in A_{\varepsilon}(S)} U_{j} \right)\\
&\geq \sum_{a_{j}\in A_{\varepsilon}(S)}\pi\left(\frac{1}{4R_{l}^{\frac{1}{M}}}|b_{j}|\right)^{2}\\
&=\pi \frac{1}{16R_{l}^{\frac{2}{M}}}\sum_{a_{j}\in A_{\varepsilon}(S)}|b_{j}|^{2}.
\end{align*}
On the other hand with \eqref{ex} there exists a $\delta>0$ such that
\begin{align*}
\area A(S)&\leq (1+\delta) \area A_{\varepsilon}(S)\\
&\leq (1+\delta)\area\left(\bigcup_{a_{j}\in A_{\varepsilon}(S)} W_{m, \eta}\right)\\
&\leq \pi(1+\delta)\tau^{2}\sum_{a_{j}\in A_{\varepsilon}(S)}|b_{j}|^{2}.
\end{align*}
We conclude that
\begin{equation}\label{tau}
\dens \left(\bigcup_{j\in\mathbb{N}}U_{j}, A(S)\right)\geq \frac{1}{16(1+\delta)\tau^{2}R_{l}^{\frac{2}{M}}}.
\end{equation}
Now consider $g_{j}$ as defined in \eqref{g}, which is a branch  of $f^{-1}$ mapping $$A'(S)=A(S)\setminus(-2S,-S)$$ into $U_{j}.$  
With $\lambda=1/2$ in \eqref{cube} and \eqref{glambda} we obtain
$$\frac{4|b_{j}|}{27m_{j}(2S)^{1+\frac{1}{m_{j}}}}\leq |g'_{j}(z)|\leq \frac{12|b_{j}|}{m_{j}S^{1+\frac{1}{m_{j}}}},$$
for $z\in A'(S).$ Then
$$\sup_{u,v\in A'(S)}\left|\frac{g'_{j}(u)}{g'_{j}(v)}\right|\leq 324,$$
for $S$ large enough. Hence with \eqref{tau} it yields 
\begin{equation}\label{jj}
\dens \left(g_{j}\left(\overline{E_{1}}\right),g_{j}\left(A'(S)\right)\right)\geq \frac{1}{324^{2}}\dens \left(\overline{E_{1}}, A'(S)\right)\geq \frac{\alpha}{ R_{l}^{2/M}},
\end{equation}
where $\alpha=1/(16(1+\delta)324^{2}\tau^{2}).$

Now we let $S=2^{k}R_{0}$ with $k\geq 0.$ Applying the above for all such $S$ and for all branches $g_{j}$ mapping $A'(S)$ to $U_{j}$ from \eqref{jj} we deduce that
\begin{equation}\label{dens2}
\dens(\overline{E_{2}},U_{j})\geq \frac{\alpha }{R_{l}^{2/M}},
\end{equation}
for each $U_{j}$ in $E_{1}.$ 

Suppose that $V\subset U_{j_{1}}$ is a component of $E_{l}.$ Let $j_{2},\cdots,j_{l}$ be such that $f^{m}(V)\subset U_{j_{m+1}}$ for $m\in\mathbb{N},$ $0\leq m \leq l-1.$ Then $f^{l-1}(V)=U_{j_{l}}$ and
\begin{equation}\label{e2}
f^{l-1}\left(\overline{E}_{l+1}\cap V\right)= \overline{E_{2}}\cap U_{j_{l}}.
\end{equation}
Denote by $g_{j_{l}}$ a branch of $f^{-1}$ that maps $U_{j_{l}}$ into $U_{j_{l-1}}.$ For $l$ large $g_{j_{l}}$ extends univalently to a map from $D\left(a_{j_{l}}, \frac{5}{6}|a_{j_{l}}|\right)$ into $B(R_{l}).$ It implies that the branch of the inverse of $f^{l-1}$ which maps $U_{j_{l}}$ to $V$ extends univalently to $D\left(a_{j_{l}}, \frac{5}{6}|a_{j_{l}}|\right).$  

Noting that $U_{j_{l}}\subset D\left(a_{j_{l}}, \frac{1}{2}|a_{j_{l}}|\right)$ by \eqref{supradiusa}, we can now apply Koebe's distortion theorem with $\lambda=\frac{3}{5}.$ From \eqref{dens2}, \eqref{e2} and  \eqref{cube} we obtain
$$\dens \left(\overline{E}_{l+1}, V\right)\geq \frac{1}{256}\dens(\overline{E}_{2}, U_{j_{l}})\geq \frac{1}{256}\frac{\alpha}{ R_{l}^{2/M}}.$$
Together with \eqref{denschor} and $B=\alpha/(9^{2}\cdot256)$ we conclude that 
\begin{equation}\label{deltal}
\dens_{\chi}\left(\overline{E}_{l+1},V\right)\geq \frac{B}{R_{l}^{2/M}}=:\Delta_{l}.
\end{equation}

Next set $h(t)$ be as in \eqref{h} and $g(t)=\left(\log^{n}\frac{1}{t}\right)^{\gamma}$.  It is easy to see that $g(t)$ is a decreasing continuous function and $\lim_{t\rightarrow0}t^{2}g(t)=\lim_{t\rightarrow 0} h(t)=0.$ Now we shall apply Lemma \ref{mass} with $h(t)$ and $g(t).$

From \eqref{dl}  we have
\begin{align*}
\log^{n}\frac{1}{d_{l}}
&= \log^{n}\left(\prod_{k=1}^{l}\frac{1}{A}q'(R_{k-1})R_{k-1}^{1+\frac{1}{M}}\right)\\
&=\log^{n-1}\log \left(\prod_{k=1}^{l}\frac{1}{A}q'(R_{k-1})R_{k-1}^{1+\frac{1}{M}}\right)\\
&=\log^{n-1}\sum_{k=1}^{l}\left(\log\frac{1}{A}+\log q'(R_{k-1})+\left(1+\frac{1}{M}\right)\log R_{k-1}\right).
\end{align*}
\\
Noting that by \eqref{qpri},
\begin{align*}
\log q'(r)
%&=\log ((\exp^{n}r^{\rho_{}})')\\
&=\log \left((\exp^{n}r^{\rho_{}})(\exp^{n-1}r^{\rho_{}})\cdots(\exp r^{\rho_{}})\rho_{} r^{\rho_{}-1}\right)\sim \exp^{n-1}\left(r^{\rho_{}}\right)
\end{align*}
as $r\rightarrow \infty,$ we have
$$
\log^{n}\frac{1}{d_{l}}\sim \log^{n-1}\sum_{k=1}^{l}\log q'(R_{k-1})\sim \log^{n-1}\sum_{k=1}^{l}\exp^{n-1}(R_{k-1}^{\rho})\sim R_{l-1}^{\rho_{}}
$$
as $l \rightarrow \infty.$
Since $R_{l}=R_{0}\exp(2^{l})$ we obtain
$$g(d_{l})=\left(\log^{n}\frac{1}{d_{l}}\right)^{\gamma}\sim R_{0}^{\gamma \rho}\exp\left(\rho_{} \gamma 2^{l-1}\right),\qquad l\rightarrow \infty.$$
Thus there exists a constant $K>0$ such that
\begin{equation}\label{gdl}
g(d_{l})\geq KR_{0}^{\gamma \rho}\exp\left(\rho_{} \gamma 2^{l-1}\right)
\end{equation}
for $l$ large.

On the other hand since $R_{l}$ is nondecreasing it follows that
 
\begin{align*}
\prod_{k=1}^{l}\Delta_{k}=
B^{l}\prod_{k=1}^{l}R_{k}^{-2/M}
&=B^{l}R_{0}^{-2l/M}\prod_{k=1}^{l}\exp \left(-\frac{2}{M} 2^{l}\right)\\
&=B^{l}R_{0}^{-2l/M}\exp \left(-\frac{2}{M}\sum_{k=1}^{l}2^{k}\right)\\
&=B^{l}R_{0}^{-2l/M}\exp \left(-\frac{8}{M}2^{l-1}+\frac{4}{M}\right).
\end{align*}
Together with \eqref{gdl} we have
\begin{equation}\label{prodgp}
g(d_{l})\prod_{k=1}^{l}\Delta_{k}\geq KB^{l}R_{0}^{\gamma\rho-2l/M}\exp\left(\frac{4}{M}\right)\exp \left(2^{l-1}\left(\rho_{}\gamma-\frac{8}{M}\right)\right)\rightarrow \infty
\end{equation}
as $l\rightarrow \infty$ if $\gamma > 8/(M\rho).$ 

With Lemma \ref{mass} and \eqref{prodgp} we complete the proof.\\

\noindent{\bf Acknowledgements.}\,\,My gratitudes are due to Prof. Dr. W. Bergweiler -my PhD supervisor, who introduced me Complex Dynamics, taught me Hausdorff measure and suggested this problem, while offering constant discussions of great help. And also to China Scholarship Council for its financial support.

The Author would like to thank the referee for his/her constructive comments.

\end{document}